\documentclass{amsart}
\usepackage{amssymb,amscd,verbatim, amsthm, amsmath, comment, multirow}
\usepackage[justification=centering]{caption}
\theoremstyle{definition}
\newtheorem{theorem}{Theorem}[section]

\newtheorem{prop}[theorem]{Proposition}
\newtheorem{conjecture}[theorem]{Conjecture}
\newtheorem{lemma}[theorem]{Lemma}
\newtheorem{corollary}[theorem]{Corollary}

\newcommand{\Q}{\mathbb{Q}}
\newcommand{\Z}{\mathbb{Z}}
\newcommand{\N}{\mathbb{N}}
\newcommand{\gpek}{g_{p,e}(k)}
\newcommand{\gzpk}{g_{p,1}(k)}
\newcommand{\gwar}[3]{g_{#1,#2}(#3)}

\newcommand{\ol}{{\mathcal O}_L}

\usepackage{color}
\usepackage[dvipsnames]{xcolor}

\title{Waring numbers of ramified $p$-adic rings}
\author[Lucas Anthony et al.]{Lucas Anthony \and Joe Burton \and Irene Deegbe \and Sarah England \and Spencer Hamblen \and Reagan Knowles \and Luke Stewart \and Hannah Wright}

\begin{document}

\begin{abstract}
Let $p$ be prime, $e,k$ be positive integers, and let $R = \Z_p[\sqrt[e]{p}]$.   We calculate the Waring numbers $g_R(k)$ for many values of $p, e$, and $k$, and investigate how the Waring numbers for $p=2$ change as $e$ and $k$ vary.
\end{abstract}

\maketitle

\section{Introduction} \label{intro}

\begin{theorem}[Waring Problem/Hilbert-Waring Theorem]
For every integer $k \geq 2$ there exists a smallest positive integer $g(k)$ such that every positive integer can be written as the sum of at most $g(k)$ $k$-th powers of integers.
\end{theorem}

The idea behind Waring Problem -- examining sums of powers -- can be easily extended to any ring.   For an excellent and thorough exposition of the research on Waring Problem and its generalizations, see Vaughan and Wooley \cite{wooley}.  We will focus on the setting of totally ramified $p$-adic rings.

Let $R$ be a ring, $k > 1$ be a positive integer, and let $R^k$ be the additive semigroup generated by all $k$-th powers of elements in $R$.  Then we define the Waring number $g_{R}(k)$ to be the smallest positive integer such that every element of $R^k$ can be written as the sum of at most $g_{R}(k)$ $k$-th powers.  A lot is known for $p > 2$ when $R$ is $\Z_p$ or an unramified extension of $\Z_p$, both general bounds (\cite{dodson1}) and specific values (\cite{bhask}, \cite{bovey}, \cite{voloch}).  These results rely heavily on work from Waring Problem over finite fields; for a summary see Demiro\u{g}lu Karabulut (\cite{demi}).

For $e \in \N$ and $p$ prime, we examine $R = \Z_p[\sqrt[e]{p}]$, the ring of integers of a totally ramified extension of $\Q_p$.  We then let $g_{p,e}(k) = g_{\Z_p[\sqrt[e]{p}]}(k)$ be the Waring numbers for these rings.  Note that $g_{p,1}(k) = g_{\Z_p}(k)$; we will more often use the former notation for consistency.  Our goal is to examine how $\gpek$ varies as $p$, $e$, and $k$ all vary. 

The methods are similar for all of these results: we first determine which elements are $k$-th powers via Hensel's Lemma (Lemma \ref{hl}) and Corollary \ref{hencal}, demonstrate how to write all elements of $R$ as sums of at most $g_{p,e}(k)$ $k$-th powers, and prove that a specific element cannot be written using fewer $k$-th powers.  We start with some general results, including when $p$ does not divide $k$ (Section \ref{gensec}).  Our most interesting results mainly focus on $p=2$ (Section \ref{2sec}), though we have several specific results for $p=3$ (Section \ref{3sec}).

For $p=2$, we have 3 main observations.  Let $\nu_2(n)$ be the 2-adic valuation of $n$, the exponent of the highest power of 2 dividing $n$.
\begin{itemize}
    \item As $\nu_2(e)$ increases, the Waring numbers decrease.  For example, $g_{2,1}(6) = 8$, $g_{2,2}(6)  = 6$, and $g_{2,4}(6) = 4$.
    \item As $\nu_2(k)$ increases, the Waring numbers increase.  For example, $g_{2,2}(3) = 2$, $g_{2,2}(2)  = 3$, and $g_{2,2}(4) = 7$.
    \item For fixed $\nu_2(e)$ and $\nu_2(k)$, as $k$ increases, the Waring numbers increase, but are fixed once $k$ is greater than a constant dependent on $e$.  For example, $g_{2,3}(2) = 4$, $g_{2,3}(6) = 6$, and $g_{2,3}(10) = g_{2,3}(k) = 8$ for all $k \equiv 2 \bmod 4$ with $k \geq 10$.
\end{itemize}

Some, but not all, of this behavior has been noted for $p=3$, where the Waring numbers seem to be affected by both divisibility by 3 and divisibility by 2.

In our preliminary investigations, it seemed that $\gpek \neq g_{p,1}(k)$ if and only if $p$ divided both $e$ and $k$.  However, neither direction of the implication holds in general, as we have 
\begin{equation}
    \gwar{3}{3}{6} = g_{\Z_3}(6) = 9, \text{ and}
\end{equation}
\begin{equation}
    \gwar{2}{3}{6} = 6 \neq g_{\Z_2}(6) = 8.
\end{equation}

It is unclear if any similar strong results will hold universally, but we have established some criteria examining when $\gpek \neq g_{p,1}(k)$ will hold.  (See Corollary \ref{2oddk}, Conjecture \ref{2evenk}, and Theorem \ref{2e2}.) 

Kowalczyk and Miska (\cite{kowmis}) calculated some values of $\gpek$ in terms of general henselian rings, with a more global focus.  For example, Example 3.27 of \cite{kowmis} proves every element of $(\Z_2[\sqrt{2}])^4$ can be written as the sum of 7 fourth powers, here replicated as Theorem \ref{224}.

\section{General results} \label{gensec}

Hensel's Lemma over extensions of $\Q_p$ is a necessary tool in every case to reduce the amount of required computation.

\begin{lemma}[Hensel's Lemma (Theorem 9.1 of \cite{conrad})] \label{hl}
Let $L$ be an extension of finite degree over $\Q_p$ with  absolute value $|\cdot|_p$, and let $\ol = \{x \in L \mid |x|_p \leq 1\}$ be the ring of integers in $L$.  Let $f(x)$ be a polynomial with coefficients in $\ol$ and $a \in \ol$ be such that
\[|f(a)|_p < |f'(a)|_p^2.\]
Then there exists $\alpha \in \ol$ such that $f(\alpha) = 0$ and $|\alpha - a|_p < |f'(a)|_p$.
\end{lemma}

Hensel's Lemma leads to a number of immediate results simplifying the calculation of $\gpek$. 
For the rest of the section, we let $L$ be a (possibly trivial) totally ramified extension of $\Q_p$, $\ol$ be the ring of integers of $L$, and $\pi_L$ be a uniformizer of $L$.  

We will use the following Corollary of Hensel's Lemma frequently in our calculations of $k$-th powers, especially in Sections \ref{2sec} and \ref{3sec}.

\begin{corollary} \label{hencal}
    Let $\nu_{L}$ be the normalized valuation of $L$, so that $\nu_L(\pi_L) = 1$.  Suppose that $k$ is a positive integer, $\alpha \in \ol$ is a unit, and there exists $\beta \in \ol$ such that $\alpha \equiv \beta^k \bmod \pi_L^{2\nu_L (k) + 1}$.  Then $\alpha$ is a $k$-th power of an element of $\ol$.
\end{corollary}

\begin{proof}
    Note that since $\alpha$ is a unit, so is $\beta$.  Let $f(x) = x^k - \alpha \in \ol[x]$.  Then $f(\beta) \equiv 0 \bmod \pi_L^{2\nu_L (k) + 1}$ and
\[|f'(\beta)|_{p} = |k\beta^{k-1}|_{p} = |k|_{p} = \frac{1}{p^{\nu_L (k)}}.\]
Therefore 
\[|f(\beta)|_{p} \leq \frac{1}{p^{2\nu_L (k) + 1}} < \frac{1}{p^{2\nu_L (k)}} = |f'(\beta)|_{p}^2,\]
and by Hensel's Lemma there exists $a \in \ol$ such that $f(a) = 0$; that is, $a^k = \alpha$.
\end{proof}

Our first result reduces the case when $p$ does not divide $k$ to the calculation of the corresponding mod $p$ Waring's number.\

\begin{lemma} \label{pnotk1}
Suppose $p$ does not divide $k > 1$. Let $g_0$ be the minimal number of non-zero $k$-th powers in $\Z/p\Z$ that sum to 0.  (Note that $g_0 \leq p$.) Then every element of $\pi_L \ol$ can be written as the sum of at most $g_0$ $k$-th powers, and $\pi_L$ cannot be written as the sum of fewer than $g_0$ $k$-th powers.
\end{lemma}
 
\begin{proof}
Take $z \in \pi_L \ol$; we start by showing that $z$ can be written as the sum of $g_0$ $k$-th powers.  Let $y_1, \dots, y_{g_0} \in \{1, \dots, p-1\}$ be such that 
\[y_1^k + \dots + y_{g_0}^k \equiv 0 \bmod p,\]
and let $f(x) = x^k + y_2^k + \dots + y_{g_0}^k - z \in \ol[x]$.  Note that $f(y_1) \equiv 0 \bmod \pi_L$, so $|f(y_1)|_p < 1$.  Additionally, $f'(y_1) = ky_1^{k-1}$; since $y_1 \neq 0$ and $p$ does not divide $k$, we then have $|f'(y_1)|_p^2 = 1$.  So by Hensel's Lemma, there exists $\alpha \in \ol$ such that $f(\alpha) = \alpha^k + y_2^k + \dots + y_{g_0}^k - z= 0$, so $z$ can be written as the sum of $g_0$ $k$-th powers.

Next, assume that $\pi_L$ can be written as the sum of fewer than $g_0$ $k$-th powers in $\ol$.  So $\pi_L = y_1^k + y_2^k + \dots + y_{g_0 - 1}^k$ for some $y_i \in \ol$.  But then $y_1^k + y_2^k + \dots + y_{g_0 - 1}^k \equiv 0 \bmod \pi_L$, so by the minimality of $g_0$, we must have $y_1 \equiv \cdots \equiv y_{g_0 - 1} \equiv 0 \bmod \pi_L$.  So $\pi_L$ divides $y_i$ for all $1 \leq i \leq g_0 - 1$, and therefore $\pi_L^k$ divides $y_i^k$ for each $i$.  This then implies $\pi_L^k$ divides $y_1^k + y_2^k + \dots + y_{g_0 - 1}^k = \pi_L$, which is a contradiction since $\pi_L$ is a uniformizer of $L$  and $k > 1$.  So $\pi_L$ cannot be written as the sum of fewer than $g_0$ $k$-th powers in $\ol$.
\end{proof}

It is worth noting that we do not always have $g_{p,1}(k) = g_0$.  For example, when $p=7$ we have $g_0 = 2$, but $g_{\Z_7}(3) = 3$.

\begin{prop} \label{pndivk}
Let $\ol$ be the ring of integers of a totally ramified extension $L$ of $\Q_p$ of degree $e$.  If $p$ does not divide $k$, then we have $g_{\ol}(k) = \gzpk$.
\end{prop}

\begin{proof}
Take $z \in \ol$.  If we let $\pi_L$ be a uniformizer of $L$, we can write $z = \sum^{\infty}_{i=0} c_i \pi_L^i$ for some $c_i \in \{0,\dots,p-1\}$.  We then let $\gamma = \sum^{\infty}_{i=0} c_i p^i$.  We claim that we can write $z$ as a sum of $n$ $k$-th powers of elements of $\ol$ if and only if we can write $\gamma$ as a sum of $n$ $k$-th powers of elements of $\Z_p$.  Note that the proof below only depends on the coefficients $c_i$; we therefore only provide the forward direction of the proof, as the reverse direction immediately follows.

Suppose we can write $z$ as the sum of $n$ $k$-th powers of elements of $\ol$, so that there exist $y_1, \dots, y_n \in \ol$ such that $z = y_1^k + \dots + y_{n}^k$.  For $j \geq 0$, let $y_{i,j} \in \{0, \dots p-1\}$ be the coefficient of $\pi_L^i$ in $y_j$, so that $y_j = \sum^{\infty}_{i=0} y_{i,j} \pi_L^i$.  Note that 
\[z \equiv \gamma \equiv c_0 \equiv y_{0,1}^k + \dots + y_{0,n}^k \bmod p.\]

We then have two cases: either there exists $j$ such that $y_{0,j} \neq 0$, or $y_{0,j} = 0$ for all $j$.  In the first case, we can, after possibly renumbering, assume that $y_{0,1} \neq 0$.  We are then in the same situation as the setup in the proof of Lemma \ref{pnotk1}, so again
by Hensel's Lemma there exists $\beta \in \Z_p$ such that 
\[f(\beta) = \beta^k + y_{0,2}^k + \dots + y_{0,n}^k - \gamma = 0,\]
which shows that $\gamma$ can be written as the sum of $n$ $k$-th powers.  

Suppose then that we are in the second case: that $y_{0,j} = 0$ for all $j$.  Let $m$ be the least positive integer such that $y_{m,j} \neq 0$ for some $j$; after possibly renumbering, we can then assume that $y_{m,1} \neq 0$ and that $y_j \equiv 0 \bmod \pi_L^{m}$ for all $j$.  Note that this implies $z \equiv 0 \bmod \pi_L^{mk}$.  We then have
\begin{equation} \label{zdivm}
    \frac{z}{\pi_L^{mk}} = \left(\frac{y_1}{\pi_L^m}\right)^k
+ \cdots + \left(\frac{y_n}{\pi_L^m}\right)^k.
\end{equation}
Since $\frac{y_j}{\pi_L^m} = \sum^{\infty}_{i=m} y_{i-m,j} \pi_l^{i-m}$ and $y_{m,1} \neq 0$, Equation (\ref{zdivm}) reduces to the proof of the first case.  After again applying Hensel's Lemma, we have that $\frac{\gamma}{p^{mk}}$, and therefore $\gamma$, is the sum of $n$ $k$-th powers.
\end{proof}

Given the previous two results, we will focus on the case when $p$ divides $k$ for the rest of the paper.

\section{The case $p=2$} \label{2sec}

We start with some general results for $p=2$; these, combined with the results above, solve all cases with odd $k$ and any $e$.

We will frequently use Hensel's Lemma here to determine the unit squares in $\Z_p[\sqrt[e]{p}]$.  Corollary \ref{hencal} states that if $\omega = \sqrt[e]{p}$ then the unit $\alpha \in \Z_p[\omega]$ is a $k$-th power if and only if $\alpha$ is congruent to a $k$-th power mod $\omega^{2\nu_{\omega}(k) + 1}$.

\begin{lemma} \label{odd2adic}
For all odd integers $k \geq 3$, we have $g_{\Z_2} (k) = 2.$  
\end{lemma}

\begin{proof}
Let $k \geq 3$ be odd.  Note that $2$ is not a $k$-th power in $\Z_2$, so we must have $g_{\Z_2} (k) \geq 2$.
We then take $\alpha \in \Z_2$; we will show that $\alpha$ is a sum of at most 2 $k$-th powers in $\Z_2$.

First, suppose that $\nu_2(\alpha) = 0$; that is, that $\alpha \equiv 1 \bmod 2$.  Here $\nu_2(k) = 0$, and so $2^{2\nu_2(k) + 1} = 2^1$.  By Corollary \ref{hencal}, since 1 is a $k$-th power, $\alpha$ is therefore itself a $k$-th power.
Then suppose that $\nu_2(\alpha) > 0$; that is, 2 divides $\alpha$. Then by the above case we have that $\alpha - 1$ and 1 are both $k$-th powers in $\Z_2$, and $\alpha = 1 + (\alpha - 1)$.

Therefore we have that for all odd integers $k \geq 3$ every element of $\Z_2$ can be written as the sum of two $k$-th powers, so $g_{\Z_2} (k) = 2.$
\end{proof}

\begin{corollary} \label{2oddk}
    For all odd integers $k\geq 3$ and all integers $e \geq 1$, $g_{2,e}(k) = 2$.
\end{corollary}

\begin{proof}
    This follows directly from Lemma \ref{odd2adic} and Proposition \ref{pndivk}.
\end{proof}

When $k$ is even, the behavior of $g_{p,e}(k)$ appears to depend on the relative sizes of $\nu_p(e)$, $\nu_p(k)$, and $k$.  Notably, though, for $p=2$, every case we have calculated so far satisfies the following conjecture.

\begin{conjecture} \label{2evenk}
If $k$ is even and $\gcd(e,k) > 1$, then $g_{2,e}(k) < g_{\Z_2} (k)$; otherwise, $g_{2,e}(k) = g_{2,1} (k)$.
\end{conjecture}

We give the proofs for specific cases that provide evidence for the observed behavior listed in Section \ref{intro}.  All known values for $p=2$ and even $k$ are given in Table \ref{gp=2}.

\begin{table}[h]
    \centering
  \caption{Waring numbers $g_{p,e}(k)$ for $p=2$}
$\begin{array}{|c|c|c|c|c|c|} \cline{3-6}
   \multicolumn{2}{c|}{} & \multicolumn{4}{|c|}{k}\\
   \cline{3-6}  \multicolumn{2}{c|}{} & 2 & 4 & 6 & k \equiv 2 \bmod 4, k \geq 10\\ \hline
        \multirow{6}{*}{$e$} & 1 & \; 4 \; & \; 15 \; & \; 8 \; & \; 8 \; \\
        \cline{2-6}  & 2 & 3  &  7  & 6 & \; 6 \; \\
        \cline{2-6} & 3 & 4  & -  &  6 & \; 8 \;  \\
        \cline{2-6} & 4 & 3  & -  & 4 & \; - \; \\
        \cline{2-6} & e \text{ odd} & 4  & -  & - & \; - \; \\
        \cline{2-6} & e \text{ even} & 3  & -  & - & \; - \; \\
        \hline
      \end{array}$
      
    \label{gp=2}
\end{table}

We are particularly interested in ``horizontal'' results, where e is fixed and $k$ varies, and ``vertical'' results, where $k$ is fixed and $e$ varies.
We start with a specific case over $\Z_2$.

\begin{theorem}
    We have $g_{2,1}(4)= 15$.
\end{theorem}

\begin{proof}
First, note that if $\alpha \in \Z_2$ and $\alpha \equiv 1 \bmod 2^5$, Corollary \ref{hencal} implies that $\alpha$ is a fourth power in $\Z_2$.  
Similarly, since $3^4 \equiv 17 \bmod 32$, if $\alpha \in \Z_2$ is such that $\alpha \equiv 17 \bmod 32$, then $\alpha$ is again a fourth power in $\Z_2$.  These are the only unit fourth powers in $\Z_2$, so we have that a unit $\alpha \in \Z_2$ is a fourth power if and only if $\alpha \equiv 1 \bmod 16$.

Then, all fourth powers in $\Z_2$ that are not units must be divisible by $2^4 = 16$; this immediately gives us that $g_{2,1}(4) \geq 15$, since 15 itself cannot be written in fewer than 15 fourth powers.  We then claim that if $\alpha \in \Z_2$ with $\nu_2(\alpha) \geq 4$, then $\alpha$ can be written as the sum of at most 15 fourth powers; this will show $g_{2,1}(4)= 15$.

Let $r = \lfloor \frac{\nu_2(\alpha)}{4} \rfloor$, and let $\alpha' = \frac{\alpha}{2^{4r}}$.  Then $\alpha' \not\equiv 0 \bmod 16$, so by the above we can write $\alpha'$ as the sum of at most 15 fourth powers.  So if
\[\alpha' = x^4 + y^4 + \dots ,\]
we then have
\[\alpha = (2^rx)^4 + (2^ry)^4 + \dots, \]
and so $\alpha$ as well can be written as the sum of at most 15 fourth powers.
\end{proof}

We will use this last method for dealing with the 0 residue class in Lemma \ref{oddub}.   We now proceed to the three general horizontal results, for fixed $e$ and varying $k$.

\begin{theorem}
If $k \equiv 2 \bmod 4$ and $k \geq 6$, then $g_{2,1} (k) = 8$.
\end{theorem}

\begin{proof}
Suppose $k \equiv 2 \bmod 4$ and $k \geq 6$. We first claim that a unit $\alpha$ in $\Z_2$ is a $k$-th power if and only if $\alpha \equiv 1 \bmod 8$.  
First, suppose $\alpha \equiv 1 \bmod 2^3$. Since $2\nu_2(k) + 1 = 3$, and 1 is a $k$-th power, Corollary \ref{hencal} implies that $\alpha$ is a $k$-th power in $\Z_2$.  
Conversely, if $x = \sum_{i=0}^{\infty} x_i 2^i$ is a unit in $\Z_2$, then 
\[x^k = 1 + 2kx_1 + 2k(k-1)x_1 + 2^2kx_2 + O(2^3).\]
(Note that for a prime $\pi$ and a positive integer $\ell$, $O(\pi^{\ell})$ implies the two sides are equal up to a multiple of $\pi^{\ell}$.)  Since $k$ is even, then $2^3$ divides both $2^2k$ and $2k + 2k(k-1) = 2k^2$, so $x^k \equiv 1 \bmod 2^3$.

This immediately implies that every element of $\Z_2$ is a sum of at most 8 $k$-th powers, since we can write every residue class mod 8 as a sum of 1s.  Note then that if $2$ divides $x$ in $\Z_2$, then since $k \geq 6$, we have that $2^6$ divides $x^6$.  Therefore to write any $a \equiv 8 \bmod 16$ in $\Z_2$, we need (at least) 8 units to write $a$ as a sum of $k$-th powers.  Therefore $g_{2,1}(k) = 8$.
\end{proof}

When we move from $\Z_2$ to $\Z_2[\sqrt{2}]$, we gain more squares, such as $2 = (\sqrt{2})^2$ and $3 + 2\sqrt{2} = 1 + (\sqrt{2})^2 + (\sqrt{2})^3$; this allow us to write elements of $(\Z_2[\sqrt{2}])^{k}$ using fewer squares.

\begin{theorem} \label{222k}
If $k \equiv 2 \bmod 4$ and $k \geq 6$, then $\gwar{2}{2}{k} = 6$.
\end{theorem}

\begin{proof}
First, note here that $2\nu_{\sqrt{2}}(k) + 1 = 5$, so by Corollary \ref{hencal} we only need to work with $k$-th powers mod $(\sqrt{2})^5$.  Then, the $k$-th power formula in $\Z_2[\sqrt{2}]$ is, for $x = \sum_{i=0}^{\infty} x_i2^i$:

\begin{align}
    x^{k} & = x_0^{k} + kx_0^{k-1}x_1(\sqrt{2}) + \left(\frac{k(k-1)}{2}x_0^{k-2}x_1^2 + kx_0^{k-1}x_2\right) (\sqrt{2})^2 + \dots \notag\\
    & = x_0 + x_0x_1 (\sqrt{2})^2 + x_0x_1(\sqrt{2})^3 + O((\sqrt{2})^5). \label{2kpower}
\end{align}    

We immediately get that the only non-zero $k$-th powers mod $(\sqrt{2})^5$ are 1 and $1+(\sqrt{2})^2 + (\sqrt{2})^3$.

We note that in Equation (\ref{2kpower}) above (and all subsequent formulas for extensions of $\Z_2$) we can ignore the exponents on coefficients, since they are either 0 or 1.  Additionally, note that since $\sqrt{2}$ does not appear in Equation (\ref{2kpower}), it cannot be the sum of $k$-th powers, and we have that 
\[(\Z_2[\sqrt{2}])^{k} = \{ x_0 + x_2 (\sqrt{2})^2 + x_3 (\sqrt{2})^3 + x_4 (\sqrt{2})^4 +\dots \mid x_i \in \{0,1\}\}\]

Table \ref{2kpowers} then demonstrates how to write any element of $(\Z_2[\sqrt{2}])^{k}$ in at most $6$ $k$-th powers.  We apply Corollary \ref{hencal} to the first term in the sum to ensure it is a $k$-th power, and note that both $8 = (\sqrt{2})^6$ and $4\sqrt{2} = (\sqrt{2})^5$ are equivalent to 0 mod $(\sqrt{2})^5$.

\begin{table}[h]
\centering
 \caption{How to write $\alpha \bmod (\sqrt{2})^5$ as sum of $k$-th powers}
\begin{tabular}{||l|l|l||}
\hline
\multicolumn{2}{|c|}{$\alpha \bmod (\sqrt{2})^5$} & \qquad $\alpha$ as a sum of $k$-th powers\\
\hline
$1$ & $1$ & $\alpha$   \\
\hline
$3$  & $1 + (\sqrt{2})^2$ & $(\alpha-2) + 1 + 1$   \\
\hline
$1 + 2\sqrt{2}$ &  $ 1 + (\sqrt{2})^3$ & $(\alpha +2 ) + (3+ 2\sqrt{2}) + (3+ 2\sqrt{2})$   \\
\hline
$5$  & $1 + (\sqrt{2})^4$ & $(\alpha - 4) + 1 + 1 + 1 + 1$   \\
\hline
$3 + 2\sqrt{2}$   & $1 + (\sqrt{2})^2 + (\sqrt{2})^3$ & $\alpha$ \\
\hline
$7$  & $1 + (\sqrt{2})^2+ (\sqrt{2})^4$ & $(\alpha -4 + 2\sqrt{2}) + (3+2\sqrt{2}) + 1$   \\
\hline
$5 + 2\sqrt{2}$  & $1 + (\sqrt{2})^3 + (\sqrt{2})^4$ & $(\alpha - 2) + 1 + 1$   \\ 
\hline
$7 + 2\sqrt{2}$  & $1 + (\sqrt{2})^2 + (\sqrt{2})^3 + (\sqrt{2})^4$ & $(\alpha - 4) + 1 + 1 + 1 + 1$  \\
\hline
$0$ & 0 & $(\alpha +3+ 2\sqrt{2}) + (3 + 2\sqrt{2}) + 1 + 1$ \\
\hline
$2$  & $(\sqrt{2})^2$ & $(\alpha - 1) + 1$ \\
\hline
$2\sqrt{2}$  & $(\sqrt{2})^3$ & $(\alpha +3) + 1 + 1 + 1 + 1 + 1$ \\
\hline
$4$  & $(\sqrt{2})^4$ & $(\alpha - 3) + 1 + 1 + 1$ \\
\hline
$2 + 2\sqrt{2}$  & $(\sqrt{2})^2 + (\sqrt{2})^3$ & $(\alpha +1) + (3 + 2\sqrt{2}) + (3 + 2\sqrt{2}) + 1$ \\
\hline
$6$  & $(\sqrt{2})^2+ (\sqrt{2})^4$ & $(\alpha - 3 + 2\sqrt{2}) + (3 + 2\sqrt{2})$ \\
\hline
$4 + 2\sqrt{2}$   & $(\sqrt{2})^3 + (\sqrt{2})^4$  & $(\alpha - 1) + 1$ \\ 
\hline
$6 + 2\sqrt{2}$  & $(\sqrt{2})^2 + (\sqrt{2})^3 + (\sqrt{2})^4$  & $(\alpha - 3) + 1 + 1 + 1$ \\
\hline
   \end{tabular}
    \label{2kpowers}
\end{table}

To complete the proof, we will show that we cannot write $2\sqrt{2} = (\sqrt{2})^3$ as a sum of fewer than 6 $k$-th powers.  Suppose then, that there exist $x,y,z,s,t \in \Z_{2}[\sqrt{2}]$ such that $x^{k} + y^{k} + z^{k} + s^{k} + t^{k} = (\sqrt{2})^3$.  Using the $k$-th power expansion in Equation (\ref{2kpower}), and letting
\begin{align*}
    \Lambda_0 & = x_0 + y_0 + z_0 + s_0 + t_0, \text{ and} \\
    \Lambda_{01} & = x_0x_1 + y_0y_1 + z_0z_1 + s_0s_1 + t_0t_1,
\end{align*}
we have
\begin{equation}
    \Lambda_0 + \Lambda_{01}(\sqrt{2})^2 + \Lambda_{01}(\sqrt{2})^3 \equiv (\sqrt{2})^3 \bmod (\sqrt{2})^5  \label{222klb}
\end{equation}
Note that $\Lambda_0 \leq 5$, and $\Lambda_{01} \leq \Lambda_0$.  We can then split Equation (\ref{222klb}) into the integral and non-integral part (even and odd powers of $\sqrt{2}$), since they are linearly independent over $\Z$.  (We will use this trick again later.)
\begin{align*}
    \Lambda_0 + \Lambda_{01}(\sqrt{2})^2 & \equiv 0 \bmod (\sqrt{2})^5  \\
    \Lambda_{01}(\sqrt{2})^3 & \equiv (\sqrt{2})^3  \bmod (\sqrt{2})^5
\end{align*}
Note that since both sides of the first congruence are integers, if they are congruent mod $(\sqrt{2})^5 = 4\sqrt{2}$, then they must be congruent mod $(\sqrt{2})^6 = 8$.  And, if we divide through the second congruence by $(\sqrt{2})^3$, we can convert both into integral congruences:
\begin{align*}
    \Lambda_0 + 2\Lambda_{01} & \equiv 0 \bmod 8,   \\
    \Lambda_{01} & \equiv 1  \bmod 2.
\end{align*}
Therefore $\Lambda_0$ must be even and non-zero, since $\Lambda_{01}$ must be positive.  But if $\Lambda_0 = 2$, the first congruence cannot be satisfied (since the second would imply $\Lambda_{01} = 1$), and if $\Lambda_0 = 4$, then $\Lambda_{01}$ would have to be even, contradicting the second congruence.  We therefore cannot write $(\sqrt{2})^3$ as a sum of 5 $k$-th powers, which combined with Table \ref{2kpowers} above gives us $g_{2,2}(k) = 6$.  \end{proof}

A similar proof shows that $g_{2,3}(6) = 6$, where we make use of the fact that in $\Z_2[\sqrt[3]{2}]$ we get that $4 = (\sqrt[3]{2})^6$ is a sixth power.

For the following two theorems, we use the $k$-th power expansion of $x^k$ in $\Z_2[\sqrt[3]{2}]$ when $k \equiv 2 \bmod 4$ and $k \geq 6$, and let $\omega = \sqrt[3]{2}$.
\begin{align}
    x^{k} & = x_0^{k} + kx_0^{k-1}x_1\omega + \left(\frac{k(k-1)}{2}x_0^{k-2}x_1^2 + kx_0^{k-1}x_2\right) \omega^2 + \dots \notag\\
    & = x_0 + x_0x_1 \omega^2 + (x_0x_2 + \delta_e x_0x_1)\omega^4 \notag \\ 
    & \qquad + (x_0x_2 + \delta_o x_0x_1)\omega^5 + (x_0x_1x_2 + \delta_o x_0x_1)\omega^6 + 
    O(\omega^7), \label{3kpower}
\end{align}  
where $\delta_e = \begin{cases} 1 & \textrm{ if } k \equiv 2 \bmod 8 \\ 0 & \textrm{ if } k \equiv 6 \bmod 8 \end{cases}$, and $\delta_o = \begin{cases} 0 & \textrm{ if } k \equiv 2 \bmod 8 \\ 1 & \textrm{ if } k \equiv 6 \bmod 8\end{cases}$.  Note that $\delta_o = 1 - \delta_e$ is the coefficient of $2^2 = 4$ in the 2-adic expansion of $k$.

\begin{theorem} \label{236}
We have $g_{2,3}(6) = 6$.
\end{theorem}

\begin{proof}
Let $\omega = \sqrt[3]{2}$.  We first determine which elements in $\Z_{2}[\omega]$ are sixth powers.  Using Corollary \ref{hencal}, since $2\nu_{\omega}(6) + 1 = 7$, we can determine using Equation (\ref{3kpower}) that every unit congruent to any of $1, 1+\omega^2+\omega^5 + \omega^6, 1+\omega^2 + \omega^4$, or $1+\omega^4 + \omega^5$ mod $\omega^7$ is a sixth power in $\Z_{2}[\omega]$.  This implies, similar to the $e=2$ case, that
\[(\Z_{2}[\omega])^6 = \{ x_0 + x_2 \omega^2 + x_3 \omega^3 + \dots \mid x_i \in \{0,1\} \}. \]

With the use of SageMath, adding all possible combinations of these sixth powers (and $4 = \omega^6$) yields a table similar to Table \ref{2kpowers} above
showing that any element $\alpha \in (\Z_{2}[\omega])^6$ can be written in at most 6 sixth powers, and that 4 residue classes, including $\omega^2 + \omega^3$, might require at least 6 sixth powers.  

To confirm that $g_{2,3}(6) = 6$, we assume 
\begin{equation} \label{236lb}
    x^6 + y^6 + z^6 + s^6 + t^6 = \omega^2 + \omega^3
\end{equation}
and seek a contradiction.  Similar to the method used in the proof of Theorem \ref{222k}, we can examine Equation (\ref{236lb}) mod $\omega^7$ and split it into 3 linearly independent integral congruences (where $\Lambda_{02}$ and $\Lambda_{012}$ are defined similarly to $\Lambda_0$ and $\Lambda_{01}$ from Theorem \ref{222k}):  
\begin{align*}
    \Lambda_0 + 4(\Lambda_{012} + \Lambda_{01}) & \equiv 2 \bmod 8, \tag{$\omega^0, \omega^3, \omega^6$}\\
    \Lambda_{02} & \equiv 0 \bmod 2, \tag{$\omega^1, \omega^4$}\\
    \Lambda_{01} + 2(\Lambda_{01} + \Lambda_{02}) & \equiv 1 \bmod 4. \tag{$\omega^2, \omega^5$}
\end{align*}

These congruences imply that $\Lambda_0 \equiv 2 \bmod 4$ and $\Lambda_{01} \equiv 3 \bmod 4$.  But this cannot happen since $0 \leq \Lambda_{01} \leq \Lambda_0 \leq 5$, by assumption.  So we have that $g_{2,3}(6) = 6$.
\end{proof}

In contrast to the situation over $\Z_2$ and $\Z_2[\sqrt{2}]$, the Waring number $g_{2,3}(k)$ increases when $k$ goes from 6 to 10.   When $k=6$, we have that 4 is a $k$-th power, which we lose for higher $k$.

\begin{theorem}
If $k \equiv 2 \bmod 4$ and $k \geq 10$, then $g_{2,3} (k) = 8$.
\end{theorem}

\begin{proof}
Let $\omega = \sqrt[3]{2}$.  We again use Corollary \ref{hencal} and Equation \ref{3kpower} to determine that any unit congruent mod $\omega^7$ to 1, $1 + \omega^2 + \omega^4$, $1 + \omega^4 + \omega^5$, or $1 + \omega^2 + \omega^5 + \omega^6$ is a $k$-th power in $\Z_2[\omega]$.   Note that here, unlike Theorem \ref{236}, we have no non-unit $k$-th powers to help us.  The subset of the sums of $k$-th powers in $\Z_2[\omega]$ is the same as Theorem \ref{236}, though:
\[(\Z_{2}[\omega])^k = \{ x_0 + x_2 \omega^2 + x_3 \omega^3 + \dots \mid x_i \in \{0,1\} \}. \]

A calculation via SageMath yields a table similar to Table \ref{2kpowers} showing that all residue classes mod $\omega^7$ can be written in at most 8 $k$-th powers.  We then claim that $8 = \omega^9$ cannot be written as the sum of 7 or fewer $k$-th powers.  So, assume
\begin{equation}x^k + y^k + z^k + q^k + r^k + s^k + t^k = 8 = \omega^9. \label{23k8}
\end{equation}

Then, for $\alpha \in \Z_2[\omega]$, let 
\[\widehat{\alpha} = \alpha_0 + \alpha_3 \omega^3 + \alpha_6 \omega^6  = \alpha_0 + 2 \alpha_3 + 4 \alpha_6 \in \{0, 1, \dots, 7\}.\]
We can then look at the integral part of Equation (\ref{23k8}), which implies
\begin{equation}
\widehat{x^k} + \widehat{y^k} + \dots + \widehat{t^k} \equiv 0 \bmod 8. \label{23k8hat}
\end{equation}

Given our list of $k$-th powers above, we then have that $\widehat{x^k}, \dots, \widehat{t^k} \in \{0,1,5\}$.  Additionally, if $\widehat{\alpha^k} = 0$, then $\omega$, and therefore also $\omega^k$, must divide $\alpha^k$, so $\alpha^k \equiv 0 \bmod \omega^{10}$.  We therefore cannot have every term on the left side of Equation (\ref{23k8hat}) equal to 0, since then the sum would be $0 \not\equiv 8 = \omega^9 \bmod \omega^{10}$.

The only other way to satisfy Equation (\ref{23k8hat}) is to have exactly 4 units.  We therefore assume without loss of generality $\widehat{x^k}, \widehat{y^k}, \widehat{z^k}, \widehat{q^k} \in \{1,5\}$ and $\widehat{r^k}, \widehat{s^k}, \widehat{t^k}  = 0$.  We then have $\omega$ dividing $r$, $s$, and $t$, so they will make no contributions to the sum mod $\omega^7$.  Then, for the sum to be 0 mod 8, the number of 1s and 5s among $\widehat{x^k}, \widehat{y^k}, \widehat{z^k}$, and $\widehat{q^k}$ must then both be odd.

Using the notation from Theorems \ref{222k} and \ref{236}, we then have $\Lambda_0 = x_0 + y_0 + z_0 + q_0 = 4$, and from the 
$\omega^2$ coefficient of Equation (\ref{3kpower}), we get that $\Lambda_{01} \equiv 0 \bmod 2$.  This then implies in the $\omega^4$ coefficient of Equation (\ref{3kpower}) we have $\Lambda_{02} + \delta_e\Lambda_{01} \equiv \Lambda_{02} \equiv 0 \bmod 2$.  We then have the following two congruences:
\begin{align}
    \Lambda_{01} + 2\Lambda_{02} & \equiv 0 \bmod 4, \tag{$\omega^2, \omega^5$} \\
    \Lambda_{012} & \equiv 0 \bmod 2. \tag{$\omega^6$}
\end{align}
Since $\Lambda_{02} \equiv 0 \bmod 2$, the first congruence above implies $\Lambda_{01} \equiv 0 \bmod 4$.  This then implies $x_1 = y_1 = z_1 = q_1$; if they are all 0, then none of $\widehat{x^k}, \widehat{y^k}, \widehat{z^k}, \widehat{q^k}$ could equal 5, a contradiction.  If $x_1 = y_1 = z_1 = q_1 = 1$, then the second congruence above implies that an even number of $x_2, y_2, z_2, q_2$ are 0 and an even number of them are 1.  But this implies that among $\widehat{x^k}, \widehat{y^k}, \widehat{z^k}$, and $ \widehat{q^k}$ we have an even number of 1s and an even number of 5s, a contradiction. Therefore $8 = \omega^9$ cannot be written as the sum of 7 $k$-th powers, and so $g_{2,3}(k) = 8$.
\end{proof}

We then move to the vertical proofs, for $k=2$ (squares) as $e$ varies.  This requires a different method for the upper bound for the Waring number, as we can no longer calculate all possible combinations of squares.

For the rest of the section, let $\omega = \sqrt[e]{2}$.

\begin{theorem} \label{2e2}
    We have $g_{2,e}(2) = \begin{cases}
        4 & \text{for odd } e \\
        3 & \text{for even } e \\
    \end{cases}$
\end{theorem}
\medskip

We treat the odd and even cases separately, and build the proof out of the Lemmas that follow.  We first note the equation for squares in each case.  When $e$ is even, we have
\begin{align}
    x^2 &= x_0^2 + x_1^2\omega^2 + x_2^2\omega^4 + \dots + x_{\frac{e}{2}}^2 \omega^e \label{evenesquare}\\
    & + x_0x_1\omega^{e+1}  + (x_0x_2 + x_{\frac{e+2}{2}}^2) \cdot \omega^{e+2} + (x_0x_3 + x_2x_1) \cdot \omega^{e+3} \notag\\
    & + (x_1x_3 + x_0x_4 + x_{\frac{e+4}{2}}^2) \cdot \omega^{e+4} + \dots \notag\\
    &+ (x_e^2 + x_0x_e + x_1x_{e-1} + \cdots + x_{\frac{e}{2} - 1}x_{\frac{e}{2} + 1}) \cdot \omega^{2e} + O(\omega^{2e+1})  \notag
\end{align}

When $e$ is odd, we have
\begin{align}
    x^2 & = x_0^2 + x_1^2\omega^2 + x_2^2\omega^4 + \dots + x_{\frac{e-1}{2}}^2\omega^{e-1} \label{oddsquare}  \\
    & + (x_{\frac{e+1}{2}}^2 + x_0x_1)\omega^{e+1} + x_0x_2 \omega^{e+2}  \notag \\
    & + (x_{\frac{e+3}{2}}^2 + x_0x_3 + x_1x_2)\omega^{e+3} + (x_0x_4+x_1x_3)\omega^{e+4} + \dots \notag \\
    & + (x_e^2 + x_0x_e + x_1x_{e-1} + \dots + +x_{\frac{e-1}{2}}x_{\frac{e+1}{2}}) \omega^{2e} + O(\omega^{2e+1}) \notag 
\end{align}

Applying Corollary \ref{hencal}, if $c \in \Z_2[\omega]$ is a unit congruent to a square mod $\omega^{2e+1}$, then $c$ is a square in $\Z_2[\omega]$.

\begin{lemma} \label{r2}
If $e \geq 2$ is even, we have 
\[R^2_{2,e} = \{\alpha_0 + \alpha_2 \omega^2 + \dots + \alpha_{e-2}\omega^{e-2} + \alpha_e\omega^e + \alpha_{e+1}\omega^{e+1} + \dots \mid \alpha_i \in \{0,1\}\}.\]

If $e \geq 3$ is odd, we have
\[R^2_{2,e} = \{\alpha_0 + \alpha_2 \omega^2 + \dots +\alpha_{e-3}\omega^{e-3} + \alpha_{e-1}\omega^{e-1} + \alpha_e\omega^e + \alpha_{e+1}\omega^{e+1} + \dots \mid \alpha_i \in \{0,1\}\}.\]
\end{lemma}

\begin{proof}
This follows directly from Equations (\ref{evenesquare}) and (\ref{oddsquare}).
\end{proof}

We then start by determining the lower bound in each case.

\begin{lemma} \label{evenlb}
    When $e$ is even, $\omega^{e+1} = 2\omega$ cannot be written as the sum of 2 squares.
\end{lemma}

\begin{proof}
Assume there exist $x,y \in \Z_2[\omega]$ such that $x^2 + y^2 = \omega^{2e+1}$.  Noting that $\omega^e = 2$, matching the coefficients of this equation and using Equation (\ref{evenesquare}) gives us $e$ linearly independent congruences (as in Theorems \ref{222k} and \ref{236}), the first three of which are
\begin{align*}
    x_0^2 + y_0^2 + (x_{\frac{e}{2}}^2 + y_{\frac{e}{2}}^2)\omega^e + \hspace{0.8in} & \tag{$\omega^0, \omega^e, \omega^{2e}$}\\
    (x_e^2 + x_0x_e + x_1x_{e-1} + \dots \hspace{0.3in} &  \\
     + x_{\frac{e}{2}-1}x_{\frac{e}{2}+1} + y_e^2 + \dots)  \omega^{2e} & \equiv 0 \bmod \omega^{2e+1}, \\
    (x_0x_1 + y_0y_1)  \omega^{e+1} & \equiv \omega^{e+1} \bmod \omega^{2e+1}, \tag{$\omega^1, \omega^{e+1}$}\\
   \hspace{0.1in}  (x_1^2 + y_1^2) \omega^2 + (x_0x_2 + x_{\frac{e+2}{2}}^2 + y_0y_2 + y_{\frac{e+2}{2}}^2) \omega^{e+2} &\equiv 0 \bmod \omega^{2e+1}. \tag{$\omega^2, \omega^{e+2}$}
\end{align*}
We assume here that $e > 2$; the case when $e=2$ is simpler and easily checked.

We can then again convert these to integral congruences to yield the following:
\begin{align*}
    x_0^2 + y_0^2 + 2(x_{\frac{e}{2}}^2 + y_{\frac{e}{2}}^2) + \hspace{2.5in}& \\
    4(x_e^2 + x_0x_e + x_1x_{e-1} + \dots + x_{\frac{e}{2}-1}x_{\frac{e}{2}+1} + y_e^2 + y_0y_e + \dots ) & \equiv 0 \bmod 8, \\
    x_0x_1 + y_0y_1 & \equiv 1 \bmod 2, \\
    x_1^2 + y_1^2 + 2(x_0x_2 + x_{\frac{2+e}{2}}^2 + y_0y_2 + y_{\frac{2+e}{2}}^2) &\equiv 0 \bmod 4.
\end{align*}
For our purposes, we only need to look at these equations mod 2:
\begin{align*}
    x_0^2 + y_0^2  & \equiv 0 \bmod 2,  \\
    x_0x_1 + y_0y_1 & \equiv 1 \bmod 2, \\
    x_1^2 + y_1^2  &\equiv 0 \bmod 2.
\end{align*}

The first equation implies that $x_0 = y_0$, and the third equation implies $x_1 = y_1$.  But then the second equation cannot be satisfied, and we get a contradiction.  Therefore $w^{e+1}$ cannot be written as the sum of 2 squares.
\end{proof}

\begin{lemma} \label{oddlb}
    When $e$ is odd, $1+\omega^{e}+\omega^{2e} = 7$ cannot be written as the sum of 3 squares.
\end{lemma}

\begin{proof}
Assume there exist $x,y,z \in \Z_2[\omega]$ such that $x^2 + y^2 + z^2 = 1 + \omega^e + \omega^{2e} =7$.    

As in the proof of Lemma \ref{evenub} above, we examine the coefficient congruences this assumption implies.  We first have (corresponding to the coefficients of powers of $\omega^e$):
\begin{align*}
    x_0^2 + y_0^2 +z_0^2 & \equiv 1 + \omega^e = 3 \bmod \omega^{e+1}. \tag{$\omega^0, \omega^e$}
\end{align*}
Since this congruence contains only integers and $\omega^{2e} = 4$, we have that $x_0^2 + y_0^2 +z_0^2 \equiv 3 \bmod 4$, so $x_0 = y_0 = z_0 = 1$.  We then have the following congruences, starting with the even powers of $\omega$ less than $e$:

\begin{align}
   (x_1^2 + y_1^2 + z_1^2)\omega^2 + (x_0x_2 + y_0y_2 + z_0z_2)\omega^{2+e} &\equiv 0 \bmod \omega^{2e+1} \tag{$\omega^2, \omega^{e+2}$} \\
    (x_2^2 + y_2^2 + z_2^2)\omega^4 + \hspace{2in}& \tag{$\omega^4, \omega^{e+4}$} \\
    (x_1x_3 + x_0x_4 + y_1y_3 + y_0y_4 + z_1z_3+z_0z_4)\omega^{4+e}& \equiv 0 \bmod \omega^{2e+1} \notag \\
    \vdots & \notag \\
    (x_{\frac{e-1}{2}}^2 + y_{\frac{e-1}{2}}^2 + z_{\frac{e-1}{2}}^2)\omega^{e-1} + \hspace{1.3in}&\tag{$\omega^{e-1}, \omega^{2e-1}$} \\
    \qquad \left(\sum_{0\leq i < \frac{e-1}{2}} (x_ix_{e-1-i} + y_iy_{e-1-i} + z_iz_{e-1-i})\right)\omega^{2e-1} & \equiv 0 \bmod \omega^{2e+1} \notag
\end{align}

Converting these into integral congruences, we have 
\begin{align}
   (x_1^2 + y_1^2 + z_1^2) + 2 (x_0x_2 + y_0y_2 + z_0z_2) &\equiv 0 \bmod 4 \notag \\
    (x_2^2 + y_2^2 + z_2^2) + 2 (x_1x_3 + x_0x_4 + y_1y_3 + y_0y_4 + z_1z_3+z_0z_4)& \equiv 0 \bmod 4 \notag \\
    \vdots & \notag \\
    (x_{\frac{e-1}{2}}^2 + y_{\frac{e-1}{2}}^2 + z_{\frac{e-1}{2}}^2) + 2 \left(\sum_{0\leq i < \frac{e-1}{2}} (x_ix_{e-1-i} + y_iy_{e-1-i} + z_iz_{e-1-i})\right) & \equiv 0 \bmod 4 \label{eveneq} 
\end{align}

This implies that $x_i^2 + y_i^2 + z_i^2$ is even for all $1 \leq i \leq \frac{e-1}{2}$.  Then, since $x_0 = y_0 = z_0 = 1$ and $x_2^2 + y_2^2 + z_2^2 = x_2 + y_2 + z_2$ is even, the first congruence above becomes
\[(x_1^2 + y_1^2 + z_1^2) \equiv 0 \bmod 4. \]
Therefore $x_1 = y_1 = z_1 = 0$. 

We will then inductively repeat this process: since $x_1 = y_1 = z_1 = 0$ and $x_4+y_4+z_4$ is even, the second congruence above implies that $x_2 = y_2 = z_2 = 0$.  For the general inductive case, assume that $x_i = y_i = z_i = 0$ for $1 \leq i \leq k-1 < \frac{e-1}{2}$ and $x_j + y_j + z_j$ is even for $k \leq j \leq 2k-2 < e-1$; we then need to show that $x_k = y_k = z_k = 0$ and that $x_{2k-1} + y_{2k-1} + z_{2k-1}$ and $x_{2k} + y_{2k} + z_{2k}$ are even.

We then need the equations for the odd powers $i$ of $\omega$ less than $e$, where there will be no corresponding $\omega^i$ term in the expansion of any square.  And, since we are only concerned with the coefficients in the expansion up to $\omega^{2e+1}$, only the coefficients of $\omega^{i+e}$ are relevant, and we only know their parity --- the coefficient sum must be even to match the coefficients of $1 + \omega^e + \omega^{2e} = 7$. Therefore for odd $i < e$, we have
\begin{equation} \label{4ke}
x_{\frac{i+e}{2}}^2 + y_{\frac{i+e}{2}}^2  + z_{\frac{i+e}{2}}^2 + \sum_{0\leq j \leq \frac{i-1}{2}} (x_jx_{i-j} + y_jy_{i-j} + z_jz_{i-j}) \equiv 0 \bmod 2.
\end{equation}

We first show that $x_{2k-1} + y_{2k-1} + z_{2k-1}$ is even.  If $2k-1 \leq \frac{e-1}{2}$, then we know $x_{2k-1} + y_{2k-1} + z_{2k-1}$ is even by Equation (\ref{eveneq}) above.  If $2k-1 > \frac{e-1}{2}$, let $m = 4k-e-2 > 0$. Note that $m$ is odd and less than $e$.  Noting that $\frac{m+e}{2} = 2k-1$, Equation (\ref{4ke}) for $\omega^{m}$ takes the form 
 \begin{equation}  \label{4kem}
x_{2k-1}^2 + y_{2k-1}^2  + z_{2k-1}^2 + \sum_{0\leq j \leq \frac{m-1}{2}} (x_jx_{m-j} + y_jy_{m-j} + z_jz_{m-j}) \equiv 0 \bmod 2.
\end{equation}
And, since $\frac{m-1}{2} = \frac{4k-e-3}{2} < \frac{2k-2}{2} = k - 1$, we have that $x_j = y_j = z_j = 0$ for all $0 \leq j \leq \frac{m-1}{2}$.  Equation (\ref{4kem}) then reduces to
\[x_{2k-1}^2 + y_{2k-1}^2  + z_{2k-1}^2 = x_{2k-1} + y_{2k-1}  + z_{2k-1} \equiv 0 \bmod 2.\]

We then examine the coefficient equation starting with $\omega^{2k}$:
\begin{align*}
    (x_k^2 + y_k^2 + z_k^2) + 2 \left(\sum_{0\leq i \leq k-1 } (x_ix_{2k-i} + y_iy_{2k-i} + z_iz_{2k-i})\right) & \equiv 0 \bmod 4,   \\
    (x_k + y_k + z_k) + 2(x_{2k} + y_{2k} + z_{2k}) &\equiv 0 \bmod 4. 
\end{align*} 
To show that $x_k = y_k = z_k = 0$, it then suffices to show that $x_{2k} + y_{2k} + z_{2k}$ is even.  If $4k < e$, then $2k \leq \frac{e-1}{2}$, and so $x_{2k} + y_{2k} + z_{2k}$ is even by our initial observations above  (under Equation (\ref{eveneq})).  If $4k \geq e$, we use the coefficient equation for $\omega^{4k-e}$.  Applying Equation (\ref{4ke}) with $i = 4k-e$ gives
\begin{equation*}
x_{2k}^2 + y_{2k}^2  + z_{2k}^2 + \sum_{0\leq j \leq \frac{4k-e-1}{2}} (x_jx_{4k-e-j} + y_jy_{4k-e-j} + z_jz_{4k-e-j}) \equiv 0 \bmod 2.
\end{equation*}

Recall that we have assumed $2k-2 < e-1$. As both these quantities are even, we also have $2k-1 < e-1$.  Therefore $4k-e < 2k$, so $\frac{4k-e-1}{2} < k$, which implies $x_j = y_j = z_j = 0$ for $1 \leq j \leq \frac{4k-e-1}{2}$.  We then have (noting $x_0 = y_0 = z_0 = 1$):
\[
x_{2k}^2 + y_{2k}^2  + z_{2k}^2 + (x_{4k-e} + y_{4k-e} + z_{4k-e}) \equiv 0 \bmod 2.
\]
Finally, since $4k-e<2k$, we know $x_{4k-e} + y_{4k-e} + z_{4k-e}$ is even, which implies $x_{2k}^2 + y_{2k}^2  + z_{2k}^2 = x_{2k} + y_{2k}  + z_{2k}$ is even.

So, we know that $x_i=y_i=z_i=0$ for $1 \leq i <\frac{e-1}{2}$.  This implies that the $(\omega^0,\omega^e,\omega^{2e})$ coefficient congruence then becomes (by Equation (\ref{oddsquare}))
\begin{align*}
    x_0^2 + y_0^2 +z_0^2 +4(x_e^2 + y_e^2 + z_e^2 + \hspace{1in} & \\
    \sum_{0 \leq i \leq \frac{e-1}{2}} x_ix_{e-j} + y_iy_{e-j} + z_iz_{e-j})& \equiv 1 + \omega^e + \omega^{2e}= 7 \bmod \omega^{2e+1} \\
    3 + 4(x_e^2 + x_e + y_e^2 + y_e + z_e^2 + z_e) & \equiv 7 \bmod \omega^{2e+1} \\
    3 & \equiv 7 \bmod \omega^{2e+1} 
\end{align*}
and we have our contradiction.
\end{proof}

We then need to prove the upper bounds for each case.  As with the lower bounds, the proofs are similar.

\begin{lemma} \label{evenub}
If $e$ is even, every element of $\Z_2[\sqrt[e]{2}]^2$ can be written as the sum of 3 squares.
\end{lemma}

\begin{proof}
Again, let $\omega = \sqrt[e]{2}$.  We first note that to prove this lemma, it suffices to show that if $\alpha \in \Z_2[\sqrt[e]{2}]^2$ with $\alpha_0 = 1$, then there exist $x,y \in \Z_2[\sqrt[e]{2}]$ such that $x^2 + y^2 = \alpha$.  Additionally, by Corollary \ref{hencal}, we only need to show that there exist $x,y \in \Z_2[\sqrt[e]{2}]$ such that $x^2 + y^2 \equiv \alpha \bmod \omega^{2e+1}$.

So, take $\alpha\in \Z_2[\sqrt[e]{2}]^2$ with $\alpha_0 = 1$.  We claim that if we take $x,y$ with $x_0 = 1$ and $y_0 = y_2 = y_4 = \cdots = y_{e-2} = 0$, the rest of the coefficients of $x$ and $y$ (up to some point) can be chosen to ensure that $x^2 + y^2 \equiv \alpha \bmod \omega^{2e+1}$.  We will prove this via induction on the $\omega^n$ coefficients of $x$ and $y$.  For our base case, we choose $x_0 = 1$ and $y_0 = 0$, which ensures $x^2 + y^2 \equiv \alpha \bmod \omega$.

We then assume that for some $1 \leq n \leq e$ we have chosen the coefficients $x_0, \dots, x_{n-1}$ and $y_0, \dots, y_{n-1}$ so that all the determined coefficients of $x^2 + y^2$ match the coefficients of $\alpha$; we will show that we can choose $x_n$ and $y_n$ such that the determined coefficients of $x^2 + y^2$ match that of $\alpha$.  This breaks down into 4 cases, depending on whether $2n \leq e$ and the parity of $n$.  We make heavy use of the formula for squares given in Equation (\ref{evenesquare}).

\underline{Case 1}: Assume that $n$ is odd.   Here we use the coefficients $\alpha_{n+e}$ (noting that $\alpha_n = 0$ by Lemma \ref{r2}) and $\alpha_{2n}$.
The coefficient of $\omega^{n+e}$ in $x^2 + y^2$ is 
\[\sum_{0\leq i \leq \frac{n-1}{2}} (x_ix_{n-i} + y_iy_{n-1}).\]
Since $x_0 = 1$ and $y_j  = 0$ for even $j$, this becomes
\[x_n + x_1x_{n-1} + \dots + x_{\frac{n-1}{2}}x_{\frac{n+1}{2}} + y_1y_{n-1} + \dots + y_{\frac{n-1}{2}}y_{\frac{n+1}{2}}.\]
All of the indices are less than $n$ except for the first one, and are therefore fixed.  We can then choose $x_n$ to ensure the coefficient matches $\alpha_{n+e}$.

Then, if $2n \leq e$, the coefficient for $\omega^{2n}$ in $x^2 + y^2$ is $x^2_{n} + y^2_{n} = x_{n} + y_{n}$.  As $x_{n}$ is fixed above, we can choose $y_n$ to ensure the coefficient matches $\alpha_{2n}$.

If $2n > e$, let $m  = 2n-e$. The coefficient for $\omega^{2n}$ in $x^2 + y^2$ is then 
\begin{equation} \label{msquare}
    x^2_{n} + y^2_{n} + \sum_{0 \leq i \leq \frac{m}{2}-1} (x_ix_{m-i} + y_iy_{m-i}).
\end{equation}
Since $n$ is odd, we have $n < e$, so $m < n$. Therefore, all of the coefficients except $y_n$ in the sum are fixed (as we have fixed $x_n$ above).  We can choose $y_n$ to ensure the coefficient matches $\alpha_{2n}$.

\underline{Case 2a}: Assume that $2n \leq e$ and that $n$ is even.  Then, since $n$ is even we let $y_n = 0$, so the coefficient of $\omega^{2n}$ in $x^2 + y^2$ is just $x^2_n = x_n$, so we can choose $x_n = \alpha_{2n}$.

\underline{Case 2b}: Assume that $2n > e$, $n < e$, and $n$ is even.  Again let $m = 2n-e$.  Then the coefficient of $\omega^{2n}$ in $x^2 + y^2$ is given above in (\ref{msquare}). 
Since $n < e$, we have $m < n$, so all of the coefficients in the sum are fixed.  Since $n$ is even, we let $y_n = 0$, and then we can choose $x_n$ to ensure the coefficient matches $\alpha_{2n}$.

\underline{Case 3}: Assume $n=e$.  The coefficient of $\omega^{2e}$ in $x^2 + y^2$ is 
\[x^2_{e} + y^2_{e} + \sum_{0 \leq i \leq \frac{e}{2}-1} (x_ix_{e-i} + y_iy_{e-i})\]
Note that $x_e$ appears twice here; we have $x^2_e + x_0x_e$ as part of the coefficient.  But $x_0 = 1$, so $x^2_e + x_0x_e = 2x_e \equiv 0 \bmod \omega^{2e}$, so $x_n$ does not affect this coefficient.  All other terms in the coefficient are fixed except for $y^2_e = y_e$, which only appears once (since $y_0= 0$), so we can choose $y_e$ so that the coefficient of $x^2 + y^2$ matches $\alpha_{2e}$.

We have now chosen $x_i$ and $y_i$ for $0 \leq i < e$ and $y_e$.  Additionally, each coefficient of $\alpha$ has been used exactly once in fixing the $x_i$ and $y_i$: the coefficients with indices of the form $n+e$ are used to determine $x_n$; the coefficients with indices of the form $2n < 2e$ are used to determine $y_n$ when $n$ is odd and $x_n$ when $n$ is even; and the $2e$ coefficient is used to determine $y_e$.  We therefore have $x^2 + y^2 \equiv \alpha \bmod \omega^{2e+1}$, so $\alpha$ can be written as the sum of 2 squares in $\Z_2[\sqrt[e]{2}]$, completing the Lemma.
\end{proof}

\begin{lemma} \label{oddub}
If $e$ is odd, every element of $\Z_2[\sqrt[e]{2}]^2$ can be written as the sum of 4 squares.
\end{lemma}

\begin{proof}
The proof follows a similar line to that of Lemma \ref{evenub}, and we frequently use Equation (\ref{oddsquare}).   Here, we first prove that if $\alpha \in \Z_2[\sqrt[e]{2}]^2$ is such that $\alpha_0 + \alpha_e \omega^e = 1$  (that is, $\alpha_0 = 1$ and $\alpha_e = 0$), then there exist $x,y \in \Z_2[\sqrt[e]{2}]$ such that $x^2 + y^2 = \alpha$.  

Take such an $\alpha$.  We will prove that if we let $x_0 = 1$ and $y_0 = y_i = 0$ for $i \in \{1, 3, \dots, e-2\}$, then we can choose $x_n$ and $y_n$ for $1 \leq n < e$ and $y_e$ such that $x^2 + y^2 \equiv \alpha \bmod \omega^{2e+1}$.  By Corollary \ref{hencal}, this is sufficient to prove the first claim. 

For our base case, we choose $x_0 = 1$ and $y_0 = 0$, which ensures $x^2 + y^2 \equiv \alpha \bmod \omega$. Then assume that for some $n$ with $1 \leq n < e$ we have chosen $x_0 , \dots, x_{n-1}$ and $y_0, \dots, y_{n-1}$.  

\underline{Case 1}: If $n<e$ is odd, note that we choose $y_n = 0$.  Then if $2n < e$, the $\omega^{2n}$ coefficient of $x^2 + y^2$ is just $x_n^2 + y_n^2 = x_n$, so we can choose $x_n = \alpha_{2n}$.  If $2n > e$, let $m = 2n-e$. Then the $\omega^{2n}$ coefficient of $x^2 + y^2$ is
\begin{equation} \label{2kodde}
    x_n^2 + y_n^2 + \sum_{0 \leq i \leq \frac{m-1}{2}} (x_ix_{m-i} + y_iy_{m-i}).
\end{equation} 

As $m$ is odd, each $y_iy_{m-i}$ term contains an odd index, so all of the $y$ coefficients drop. As $m < n$, all of the $x$ coefficients are known except $x_n^2$.  We can then choose $x_n$ to ensure the coefficient of $x^2 + y^2$ matches $\alpha_{2n}$.

\underline{Case 2}: If $n$ is even, we start by examining the $\omega^{n+e}$ coefficient of $x^2 + y^2$.  Recalling that $x_0 = 1$ and $y_0 = 0$, this coefficient is
\begin{align*}
    x_0&x_n + x_1x_{n-1} + \dots + x_{\frac{n}{2} - 1}x_{\frac{n}{2} + 1} + y_0y_n + y_1y_{n-1} + \dots + y_{\frac{n}{2} - 1}y_{\frac{n}{2} + 1} \\
    & = x_n + \left(x_1x_{n-1} + \dots + x_{\frac{n}{2} - 1}x_{\frac{n}{2} + 1} + y_1y_{n-1} + \dots + y_{\frac{n}{2} - 1}y_{\frac{n}{2} + 1}\right).
\end{align*}
We can therefore choose $x_n$ to ensure the coefficient of $x^2 + y^2$ matches $\alpha_{n+e}$.  

We then examine the $\omega^{2n}$ coefficient of $x^2 + y^2$.  If $2n < e$, this coefficient is $x_n^2 + y_n^2$. As $x_n$ is fixed above, we can choose $y_n$ to ensure the coefficient of $x^2 + y^2$ matches $\alpha_{2n}$.

If $2n > e$, the expression for the $\omega^{2n}$ coefficient is given in Formula (\ref{2kodde}) above.
Since $m < n$ and we have fixed $x_n$ above, the only unknown is $y_n^2 = y_n$. We can therefore choose $y_n$ to ensure the coefficient of $x^2 + y^2$ matches $\alpha_{2n}$.

We have now chosen $x_i$ and $y_i$ for $0 \leq i < e$. We then choose $y_e$ using the $\omega^{2e}$ coefficient as in the proof of Lemma \ref{evenub}.  Additionally (as in Lemma \ref{evenub}), each coefficient of $\alpha$ (except $\alpha_e = 0$) has been used exactly once in fixing the $x_i$ and $y_i$.  We therefore have $x^2 + y^2 \equiv \alpha \bmod \omega^{2e+1}$, so $\alpha$ can be written as the sum of 2 squares in $\Z_2[\sqrt[e]{2}]$, completing the first claim.

So, then suppose $\alpha \in \Z_2[\sqrt[e]{2}]^2$ is such that $\alpha_0 + \alpha_e \omega^e \neq 1$.  If $\alpha_0 + \alpha_e \omega^e = 2$, note that $\beta = \alpha - 1\in \Z_2[\sqrt[e]{2}]^2$ and $\beta_0 + \beta_e \omega^e = 1$.  So $\beta$ can be written as the sum of 2 squares, and therefore $\alpha = \beta + 1$ can be written as the sum of 3 squares.  Similarly, if $\alpha_0 + \alpha_e \omega^e = 3$, we can use $\alpha - 2$ to write $\alpha$ as the sum of 4 squares.

Suppose $\alpha_0 + \alpha_e \omega^e = 0$.  We then need to factor out a power of $\omega$ to write $\alpha$ as the sum of at most 4 squares.  Let $\nu_{\omega}(\alpha) = j$.  First suppose that $j$ is odd. Note that $j > e$ since $\alpha_e = 0$ and for $\alpha$ to be in $\Z_2[\sqrt[e]{2}]^2$, all coefficients of $\alpha$ with odd indices less than $e$ must be 0.  Let $m = \frac{j-e}{2}$, and let $\beta = \frac{\alpha}{(\omega^{m})^2}$.  Then $\nu_{\omega}(\beta) = e$, which implies $\beta \in \Z_2[\sqrt[e]{2}]^2$ and $\beta_0 + \beta_e \omega^e = 2$, so there exist $x,y,z \in \Z_2[\sqrt[e]{2}]$ such that $x^2 + y^2 + z^2 = \beta$ by the previous paragraph.  Therefore $\alpha = (x\omega^m)^2 + (y\omega^m)^2 + (z\omega^m)^2$.

Then suppose that $\nu_{\omega}(\alpha) = j$ is even.  Let $\ell$ be the least odd natural number such that $\alpha_{\ell} \neq 0$ (if it exists).  Let $m = \min\{j, \ell-e\}/2$. Note that $m > 0$ since $\alpha_0 + \alpha_e \omega^e = 0$.  As above, let $\beta = \frac{\alpha}{(\omega^{m})^2}$, and note that
\[\beta_0 + \beta_e \omega^e =
\begin{cases}
    1 & \text{ if } j < \ell - e, \\
    2 & \text{ if } j > \ell - e, \\
    3 & \text{ if } j = \ell - e.     
\end{cases}
\]
In each case we can write $\beta$, and therefore $\alpha$ as the sum of at most 4 squares, completing the proof of the lemma.
\end{proof}

Lemmas \ref{evenlb}, \ref{oddlb}, \ref{evenub}, and \ref{oddub} then complete the proof of Theorem \ref{2e2}.

We provide two more notable Waring numbers.  The proofs and calculations are similar to those given above.

\begin{theorem} \label{246}
We have $g_{2,4}(6) = 4$.
\end{theorem}

\begin{proof} We proceed similar to the proof of Theorem \ref{236}.

Let $\omega = \sqrt[4]{2}$. We first determine which elements in $\Z_2[\sqrt[4]{2}]$ are sixth powers. Note that by Corollary \ref{hencal}, any unit of $\Z_2[\omega]$ congruent to a sixth power mod $\omega^9$ is also a sixth power, since $2\nu_{\omega}(6) + 1 = 9$. We determine all sixth powers of units in $\Z_2[\sqrt[4]{2}]$ in the following table.

\begin{table}[h]
    \centering
    \caption{Sixth powers of units $\bmod$ $\omega^9$ in $\Z_2[\omega]$}
$\begin{array}{|l|} 
\hline
   1 \\ \hline
1 + \omega^6 + \omega^7 \\ \hline
1 + \omega^4 + \omega^6 \\ \hline
1 + \omega^4 + \omega^7 \\ \hline
1 + \omega^2 + \omega^4 + \omega^5 + \omega^8 \\ \hline
1 + \omega^2 + \omega^4 + \omega^5 + \omega^6 + \omega^7 \\ \hline
1 + \omega^2 + \omega^5 + \omega^6 + \omega^7 + \omega^8 \\ \hline
1 + \omega^2 + \omega^5 \\ \hline
\end{array}$ 
    \label{g246table}
\end{table}

This implies that
\[(\Z_{2}[\omega])^6 = \{ x_0 + x_2 (\omega^2 + \omega^5) + x_4 \omega^4 + x_6 \omega^6 + x_7 \omega^7 + \dots \mid x_i \in \{0,1\} \}. \]

A calculation via SageMath shows that any element $\alpha \in (\Z_{2}[\omega])^6$ can be written in at most 4 sixth powers. Adding all possible combinations of 3 sixth powers shows that 1 residue class, $\omega^7$, requires at least 4 sixth powers.  We therefore have that $g_{2,4}(6) = 4$.

\end{proof}

\begin{theorem} \label{224}
We have $g_{2,2}(4) = 7$.
\end{theorem}

\begin{proof}
    Let $\omega= \sqrt{2}$. Note that by Corollary \ref{hencal}, any unit of $\Z_2[\omega]$ congruent to a fourth power mod $\omega^9$ is also a fourth power, since $2\nu_{\omega}(4) + 1 = 9$. Considering the fourth power formula for $\Z_{2}[\omega]$ $\bmod$ $\omega^9$, we see
    \begin{align*}
    x^4 & = x_0^4 + (x_0^2x_1^2 + x_1^4)\omega^4 + (x_0^3x_1)\omega^5 \\
    &\qquad + (x_0^3x_2 + x_0^2x_1^2 + x_0^2x_2^2)\omega^6 
    + (x_0^3x_3 + x_0x_1^3 + x_0^2x_1x_2)\omega^7 \\
    & \qquad + (x_2^4 + x_0^3x_4 + x_0^2x_2^2 + x_0^2x_3^2 
    + x_1^2x_2^2 + x_0^2x_1x_3 + x_0x_1^2x_2)\omega^8 + O(\omega^9)
    \end{align*}
Note that any even power of $\omega$ is an integer, and so can be written as a sum of fourth powers.  We therefore have 
\[(\Z_{2}[\omega])^4 = \{ \alpha_0 + \alpha_2 \omega^2 + x \omega^4  \mid \alpha_i \in \{0,1\}, x \in \Z_2[\omega] \}. \]

The fourth powers of units mod $\omega^9$ in $\Z_2[\omega]$ are then $1 + \alpha_5 \omega^5 + \alpha_7 \omega^7 + \alpha_8 \omega^8$ for $\alpha_i \in \{0,1\}$. We also have that $4 = \omega^4$ and $16 = \omega^8$ are fourth powers.
Given the relatively large number of fourth powers (including 4), it is quick to check that every residue class mod $\omega^9$ in $(\Z_{2}[\omega])^4$ can be written as the sum of 7 fourth powers including at least one unit.  

To confirm $g_{2,2}(4) = 7$, we assume there exists $a, b, c, d, e, f \in \Z_{2}[\omega]$ such that $\omega^5 = 4 \sqrt{2} = a^4 + b^4 + c^4 + d^4 + e^4 + f^4$ and seek a contradiction. Similar to the method used in the proof of Theorem \ref{222k}, we can split this equation into the integral and non-integral parts, and then convert these equations into integral congruences mod $\omega^7 = 8\sqrt{2}$.  Defining $\Lambda_0$, $\Lambda_{01}$, and $\Lambda_1$ as in Theorems \ref{222k} and \ref{236}, we have 
\begin{align*}
    \Lambda_0 + 4(\Lambda_1 + \Lambda_{01}) + 8 \Lambda_{01}& \equiv 0 \bmod 16, \tag{\text{even powers}}\\
    4\Lambda_{01} & \equiv 4 \bmod 8, \tag{\text{odd powers}}
\end{align*}

From the even powers congruence, we immediately get that $\Lambda_0 \equiv 0 \bmod 4$.  Since $\Lambda_{0} \neq 0$ and $0 \leq \Lambda_0 \leq 6$, we must have $\Lambda_0 = 4$.  The odd powers congruence gives us that $\Lambda_{01}$ is odd, and since $\Lambda_0 = 4$ we must have $\Lambda_{01} = 1$ or $3$. 
 If $\Lambda_{01} = 1$, then $1 \leq \Lambda_1 \leq 3$, and the even powers congruence becomes
\[1 + (\Lambda_1 + 1) + 2 (1) \equiv \Lambda_1 \equiv 0 \bmod 4.\]
We therefore have $\Lambda_1=0$ or $4$, contradicting our bounds above.   If $\Lambda_{01} = 3$, then $3 \leq \Lambda_1 \leq 5$, and the even powers congruence becomes
\[1 + (\Lambda_1 + 3) + 2 (3) \equiv \Lambda_1 + 2 \equiv 0 \bmod 4.\]
We then have $\Lambda_1=2$ or $6$, contradicting our bounds in this case.
   
Thus $\omega^5 = 4 \sqrt{2}$ requires at least 7 fourth powers, completing the proof.
    
\end{proof}

\section{The case $p=3$} \label{3sec}

More is known about $g_{p,1} (k)$ when $p$ is odd; the work of Voloch \cite{voloch}, along with the work of Demiro\u{g}lu Karabulut \cite{demi} over finite fields, is essentially enough to calculate all reasonable values.  

Recall that if $3$ does not divide $k$, we have $\gwar{3}{e}{k} = g_{\Z_3}(k)$.  Table \ref{p3} shows the values we have calculated when $3$ divides $k$. 

\begin{table}[h]
    \centering
    \caption{Waring numbers $g_{p,e}(k)$ for $p=3$}
$\begin{array}{|c|c|c|c|c|} \cline{3-5}
      \multicolumn{2}{c|}{} & \multicolumn{3}{|c|}{k} \\
   \cline{3-5} \multicolumn{2}{c|}{} & 3 & 6 & 9\\ \hline\multirow{6}{*}{$e$}  & 1 & \; 4 \; & \;  9 \; & 13\\
        \cline{2-5} &  2 & 4  &  9 & 13\\
        \cline{2-5} &  3 & 3  &   9 & 13 \\
        \cline{2-5} &  4 & -  &  9 & -\\
        \cline{2-5} &  5 & -  &  9 & - \\
        \cline{2-5} &  6 & 3 &  4 &  -\\ \hline
\end{array}$ 
    \label{p3}
\end{table}

We provide proofs of some of these cases below.

\begin{theorem}
    We have $g_{3,3}(3) = 3$.
\end{theorem}

\begin{proof}
Let $\omega = \sqrt[3]{3}$. First, note that by Corollary \ref{hencal}, any unit of $\Z_3[\omega]$ congruent to a cube mod $\omega^7$ is also a cube, since $7 = 2\nu_{\omega}(3) + 1$.  Then, we have that 
\begin{align*}
        x^3 & = x_0^3 + (3x_1x_0^2)\omega + (3x_2x_0^2 + 3x_1^2x_0)\omega^2 + (3x_3x_0^2 + 6x_2x_1x_0 + x_1^3)\omega^3 + \dots \\
        & = x_0^3 + (x_1^3)\omega^3 + (x_1x_0^2)\omega^4 + (x_2x_0^2 + x_1^2x_0)\omega^5 \\
        & \hspace{1.2in} + (x_3x_0^2 + 2x_2x_1x_0 + x_2^3)\omega^6 + O(\omega^7)
\end{align*}
This implies that 
\[(\Z_3[\omega])^3 = \{\alpha_0 + x\omega^3 \mid \alpha_0 \in \{0,1,2\}, x \in \Z_3[\omega]\}.\]
Additionally, the linear $x_2$ term in the coefficient of $\omega^5$ and the linear $x_3$ term in the coefficient of $\omega^6$ in the cubic formula above imply that we only need to examine cubes mod $\omega^5$, less than the $\omega^7$ implied by Corollary \ref{hencal}: if $\alpha \equiv x^3 \bmod \omega^5$, then there will exist some $\delta_2, \delta_3 \in \{0,1,2\}$ such that $\alpha \equiv (x + \delta_2 \omega^2 + \delta_3 \omega^3)^3 \bmod \omega^7$.  For the last part of the setup, then, we check possible values of $x_i$ in the cubic equation above to see that a unit $x$ in $\Z_3[\omega]$ is a cube if and only if $x$ is equivalent mod $\omega^5$ to one of the following:
\begin{equation} \label{3cubes}
    1, 1+\omega^3 + \omega^4, 1+2\omega^3 + 2\omega^4, 2+2\omega^3, 2 + \omega^4, 2+\omega^3+2\omega^4.
\end{equation}

It is quick to check that every unit of $(\Z_3[\omega])^3$ is equivalent mod $\omega^5$ to the sum of one of the cubes above and one of $0, \omega^3, 2\omega^3$, the non-unit cubes mod $\omega^5$.  This then (by adding 1 to an element $x \equiv 2 \bmod \omega$) implies that every element of $(\Z_3[\omega])^3$ is the sum of 3 cubes.

To complete the proof, we show that $\omega^4$ cannot be written as the sum of 2 cubes.  If $x,y \in \Z_3[\omega]$ are such that $x^3 + y^3 = \omega^4$, the cubic formula implies that both $x$ and $y$ must be units.  But no combinations of 2 elements in the list of cubes above (\ref{3cubes}) sum to $\omega^4 \bmod \omega^5$, so $\omega^4$ requires at least 3 cubes, completing the proof.
\end{proof}

Voloch (\cite{voloch}) showed that $g_{\Z_3}(6) = g_{3,1}(6) = 9$.  The theorems below and preliminary calculations seem to indicate that this doesn't change until $e \geq 6$.  For reference for the next two theorems, note that for $x \in \Z_p[\omega]$, we have
\begin{align}
    x^6 & = x_0^6 + 6x_0^5x_1\omega + (15x_0^4x_1^2 + 6x_0^5x_2) \omega^2 + \label{6thp}\\
    & \qquad (6x_0^5x_3 + 20x_0^3x_1^3 +30 x_0^4x_1x_2) \omega^3 + \notag\\
    & \qquad (6x_0^5x_4 + 15x_0^4x_2^2 + 15x_0^2x_1^4 + 30x_0^4x_1x_3 + 60x_0^3x_1^2x_2)  \omega^4 + \dots\notag
\end{align} 

\begin{theorem} \label{336}
    We have $g_{3,3}(6) = 9$.
\end{theorem}

\begin{proof}
Let $\omega = \sqrt[3]{3}$. Our first step is to determine $(\Z_3[w])^6$.  Reducing the sixth power formula above (Equation (\ref{6thp})) when $\omega^3 = 3$, we get
\begin{align*}
    x^6 & = x_{0}^6 + 2x_{0}^3 x_{1}^3 \omega^3  + 2x_{0}^5 x_{1}\omega^4 + (2 x_{0}^5 x_{2} + 2x_{0}^4 x_{1}^2)\omega^5  + \\
    & \qquad (2x_{0}^5 x_{3} + 2x_{0}^3 x_{2}^3 +x_{0}^4 x_{1} x_{2} + x_{1}^6)\omega^6 + \\
    & \qquad (2x_{0}^5 x_{4} + x_0^4x_1x_3 + 2x_0^4x_2^2 + 2x_0^3x_1^2x_2 + 2x_0^2x_1^4)\omega^7 + \\
    & \qquad (2x_0^5x_5 + x_0^4x_1x_4 + x_0^4x_3x_2 +   2x_0^3x_1x_2^2 + 2x_0^3x_1^2x_3 \\
    & \hspace{1.2in} + 2x_0^2x_1^3x_2 + 2x_0x_1^5 + x_0^4x_1^2)\omega^8 + O(\omega^9)
\end{align*}

Note that for every $\omega^i$ for $i \geq 5$ the coefficient contains a term linear in $x_{i-3}$ that is the first and only appearance of $x_{i-3}$ in the $\omega^i$ coefficient.  (For example, the $2 x_{0}^5 x_{2}$ term in the coefficient of $\omega^5$.)  These terms allow us to improve the result from Corollary \ref{hencal} (which allows us to look mod $\omega^9$) to say that any unit in $\Z_3[w]$ that is congruent to a sixth power mod $\omega^5$ is a sixth power in $\Z_3[w]$.  

Our relevant sixth power formula then becomes
\begin{align*}
    x^6  & = x_{0}^6 + 2x_{0}^3 x_{1}^3 \omega^3  + 2x_{0}^5 x_{1}\omega^4 + O(\omega^5)
\end{align*}
Then, since $1+1+1 = 3 = \omega^3$, any combination of coefficients for $\omega^3$ and $\omega^4$ can occur in sums of sixth powers, so we have (again)
\[(\Z_3[w])^6 = \{\alpha_0 + x\omega^3 \mid \alpha_0 \in \{0,1,2\}, x \in \Z_3[\omega]\}.\]

We then need to show that every element of $(\Z_3[w])^6$ can be written as the sum of at most 9 sixth powers.  There are only 3 unit sixth powers mod $\omega^5$:
\begin{equation*}
    1, 1 + \omega^3 + \omega^4 = 4 + \omega^4, 1 + 2\omega^3 + 2\omega^4 = 7 + 2\omega^4
\end{equation*}
If $\alpha = \alpha_0 + \alpha_3\omega^3 + \alpha_4\omega^4 + \ldots \in (\Z_3[w])^6$, then we claim we can write $\alpha$ as a sum of at most 9 sixth powers.  Let $M = \alpha_0 -1 + 3\alpha_3 - 3\alpha_4 = \alpha_0 - 1 + (\alpha_3 - \alpha_4)\omega^3$.  Note that $-7\leq M \leq 8$.  Note that $1 + \alpha_4\omega^3 + \alpha_4\omega^4 + \sum_{i \geq 5} \alpha_i \omega^i$ is congruent mod $\omega^5$ to one of the sixth powers listed above (depending on $\alpha_4$), and is therefore itself a sixth power.  So, if $M \geq 0$, then 
\begin{align*}
    M \cdot 1 + (1 + \alpha_4\omega^3 + & \alpha_4\omega^4 + \dots) \\
    & = \alpha_0 - 1 + (\alpha_3 - \alpha_4)\omega^3 + (1 + \alpha_4\omega^3 + \alpha_4\omega^4 + \dots)\\
    & = \alpha_0 + (\alpha_3 - \alpha_4 + \alpha_4)\omega^3 + \alpha_4 \omega^4 + \dots \\
    & = \alpha,
\end{align*}
and therefore $\alpha$ can be written as the sum of $M+1 \leq 9$ sixth powers.  If $M < 0$, then (noting that -8 is a sixth power in $\Z_3[w]$ since it is congruent to 1 mod $\omega^5$) we have $\alpha = -8 + (M+8) \cdot 1 + (1 + \alpha_4\omega^3 + \alpha_4\omega^4 + \dots)$.  Since in this case $7 \geq M+8 > 0$, we again have that $\alpha$ can be written as the sum of at most $9$ sixth powers. Therefore $g_{3,3}(6) \leq 9$.

We then claim that $\omega^3 + \omega^4 = 3 + \omega^4$ cannot be written as the sum of fewer than 9 sixth powers in $(\Z_3[w])^6$, which would then imply that $g_{3,3}(6) = 9$.  Let $A = 4 + \omega^4$ and $B = 7 + 2\omega^4$.  First, note that if we write $3 + \omega^4$ as a sum of copies of $1, A$, and $B$ (the unit sixth powers in $\Z_3[w]$), the sum must include at least one copy of $A$ or $B$.  Then, note that $2B \equiv 1 + A \bmod \omega^5$, $A+B \equiv 2 \bmod \omega^5$, and $3A \equiv 3 \bmod \omega^5$. We may therefore assume without loss of generality that if we can write $3 + \omega^4$ as a sum of $N$ sixth powers (since the coefficient of $\omega^4$ is 1) that these sixth powers contain exactly 1 copy of $A$ and no copies of $B$.  But then $3 + \omega^4 \equiv A + (N-1)\cdot 1 = N+3 + \omega^4 \bmod \omega^5$.  Since $N > 0$, the least value possible for $N$ is 9, so $3 + \omega^4$ cannot be written as the sum of fewer than 9 sixth powers, completing the proof.
\end{proof}

The main difference between $\Z_3[\sqrt[3]{3}]$ and $\Z_3[\sqrt[6]{3}]$ is that 3 is now a sixth power; this reduces the Waring number for $\Z_3[\sqrt[6]{3}]$ to 4.

\begin{theorem}
We have $g_{3,6}(6) = 4$.
\end{theorem}

\begin{proof}
Let $\omega = \sqrt[6]{3}$.  Using the sixth power formula (Equation (\ref{6thp})), we get 
\begin{align*}
    x^6 & = x_{0}^6+(2x_{0}^3x_{1}^3)\omega^3+(x_{1}^6+2x_{0}^3x_{2}^3)\omega^6+ (2x_{0}^5x_{1})\omega^7  \\
    & \qquad + (2x_{0}^5x_{2}+2x_{0}^4x_{1}^2)\omega^8+(x_{0}^4x_{1}x_{2}+2x_{0}^5x_{3}+2x_{0}^3x_{3}^3)\omega^9\\
    & \qquad + (2x_0^5x_4 + f_{10}(x_0, x_1, x_2, x_3))\omega^{10} + (2x_0^5x_5 + f_{11}(x_0, x_1, x_2, x_3, x_4))\omega^{11}  \\
    & \qquad + (2x_0^5x_6 + f_{12}(x_0, x_1, x_2, x_3, x_4, x_5))\omega^{12} + O(\omega^{13}),
\end{align*}
where the $f_i$ are homogenous polynomials of degree 6.  As in the proof of Theorem \ref{336} above, the presence of the linear terms in $x_{i-6}$ in the coefficients of $\omega^i$ for $i \geq 10$ allows us to improve the result from Corollary \ref{hencal}.  Additionally, note that $2x_{0}^5x_{3}+2x_{0}^3x_{3}^3 \equiv x_0x_3 \bmod 3$, so the coefficient of $\omega^9$ can be similarly determined by $x_3$.  Consequently, any unit in $\Z_3[\omega]$ congruent to a sixth power mod $\omega^9$ is also a sixth power in  $\Z_3[\omega]$.

\begin{table}[h]
    \centering
    \caption{Sixth powers $\bmod$ $\omega^9$ in $\Z_3[\omega]$}
$\begin{array}{|l|} 
\hline
   1 \\ \hline
1 + \omega^6 + \omega^8 \\ \hline
1 + 2\omega^6 + 2\omega^8 \\ \hline
1 + \omega^3  + \omega^7 + \omega^8 \\ \hline
1 + \omega^3 + \omega^6 + \omega^7 + 2\omega^8 \\ \hline
1 + \omega^3 + 2\omega^6 + \omega^7  \\ \hline
1 + 2\omega^3 + 2\omega^7 + \omega^8  \\ \hline
1 + 2\omega^3 + \omega^6 + 2\omega^7 + 2\omega^8 
 \\ \hline
1 + 2\omega^3 + 2\omega^6 + 2\omega^7 
 \\ \hline
\omega^6 \\ \hline
\end{array}$ 
    \label{g366table}
\end{table}

Given the sixth powers listed in Table \ref{g366table}, we get
\[ (\Z_3[\omega])^6 = \{\alpha_0 + \alpha_3 \omega^3 + x \omega^6 \mid \alpha_0, \alpha_3 \in \{0,1,2\}, x \in \Z_3[\omega]\}
\]

A calculation via SageMath shows that any element $\alpha \in (\Z_3[\omega])^6$ can be written in at most 4 sixth powers. Adding all combinations of 3 sixth powers yields two residue classes, $\omega^6 + \omega^8$ and $2\omega^6 + 2\omega^8$, that cannot be written as the sum of 3 sixth powers.  (That these residue classes cannot be written as the sums of 3 sixth powers can also be deduced from the correlation of the appearance of $\omega^6$ and $\omega^8$ in the list of sixth powers in Table \ref{g366table}.)

We then have $g_{3,6}(6) = 4$.

\end{proof}

\begin{theorem}
We have $g_{3,e}(9) = 13$ for $e = 1,2,3$.
\end{theorem}

\begin{proof}
    In $\Z_3$, we have
    \[x^9 = x_0^9 + (x_0^8x_1)3^3 + (x_0^8x_2 + x_0^7x_1^2 + x_0^6x_1^3)3^4 + O(3^5).\]
    The linear $x_1$ term in the $3^3$ coefficient and the linear $x_2$ term in the $3^4$ coefficient give us that a unit $\alpha \in \Z_3$ is a ninth power if and only if $\alpha \equiv \pm 1 \bmod 27$, and there are no other ninth powers mod 27.  Since $13 \equiv -14 \bmod 27$, this immediately implies that 13 cannot be written as the sum of fewer than 13 ninth powers.  And, since 13 is the sum of 13 1s, we have $g_{\Z_3}(9) = 13$.
    
Next, let $\omega = \sqrt{3}$.   Then in $\Z_3[\omega]$, we have
\[x^9 = x_0^9 + (x_0^8x_1 + x_0^6x_1^3)\omega^5 + O(\omega^6).\]
Note then that if $x_0 \neq 0$, the $\omega^5$ coefficient is equivalent to $2x_1 \bmod 3$.  This, along with similar linear coefficients in $\omega^m$ for $m \geq 6$, imply that a unit $\alpha \in \Z_3[\omega]$ is a ninth power if and only if $\alpha \equiv \pm 1 \bmod \omega^5$.  

Since we can express every even power of $\omega$ as a sum of 1s, we then have 
\[(\Z_3[\omega])^9 = \{\alpha_0 + \alpha_2\omega^2 + x\omega^4 \mid \alpha_0, \alpha_2 \in \{0,1,2\}, x \in \Z_3[\omega]\}.\]
Since $27 = \omega^6$, the calculation of $g_{3,2}(9)$ then involves exactly the same upper and lower bounds as above in calculating $g_{\Z_3}(9)$, so $g_{3,2}(9) = 13$.

Finally, let $\omega = \sqrt[3]{3}$.   Then in $\Z_3[\omega]$, we have
\[x^9 = x_0^9 + x_0^6x_1^3\omega^6 + x_0^8x_1\omega^7 + (x_0^8x_2 + x_0^7x_1^2) \omega^8 + O(\omega^9).\]
The linear $x_2$ term in the coefficient of $\omega^8$ and similar terms in higher coefficients imply that a unit $\alpha \in \Z_3[\omega]$ is a ninth power if and only if $\alpha \equiv \pm1 + y(\omega^6 + \omega^7) \bmod \omega^8$ for some $y \in \{0,1,2\}$.  

Since we can express every third power of $\omega$ as a sum of 1s, we then have 
\[(\Z_3[\omega])^9 = \{\alpha_0 + \alpha_3\omega^3 + x\omega^6 \mid \alpha_0, \alpha_3 \in \{0,1,2\}, x \in \Z_3[\omega]\}.\]
Noting that $27 = \omega^9$, we can write every residue class mod $\omega^8$ as the sum of at most 13 ninth powers, including one or two ninth powers with an $\omega^7$ term being necessary.  However, 13 still cannot be written as the sum of fewer than 13 ninth powers, despite the promising $\omega^6 = 9$ term appearing in some ninth powers.  Indeed, the lack of an $\omega^7$ term in 13 implies that we cannot use the associated $\omega^6$ term to write 13 as the sum of fewer ninth powers.  So we have $g_{3,3}(9) = 13$.
\end{proof}

It is quite likely that $g_{3,e}(9) < 13$ for $e \geq 4$, since then $(\sqrt[e]{3})^9$ will be in a non-zero residue class mod 27, though we have not yet calculated any such values.

\section{Acknowledgements}

The authors would like to thank the McDaniel Student-Faculty Collaborative Summer Research Fund and the Lightner Fund for their support of our research.  We would also like to thank James Benjamin, Chandra Copes, Maia Hanlon, Kevin Rabidou, Luke Shuck, and Ash Wright for their contributions to this work, Tim Banks for his careful reading of the early drafts of this paper, and the anonymous referee for their helpful comments and suggestions.

\bibliographystyle{plainnat}

\begin{thebibliography}{0}

\bibitem{bhask}
M. Bhaskaran, Sums of $p$-th powers in $p$-adic rings, Acta Arith. {\bf 15} (1969), 217--219

\bibitem{bovey}
J.D. Bovey, A note on Waring's problem in $p$-adic fields, Acta Arith. {\bf 29} (1976), 343--351

\bibitem{conrad}
K. Conrad, Hensel's Lemma, {\tt https://kconrad.math.uconn.edu/blurbs/gradnumthy/hensel.pdf}

\bibitem{demi}
Y. Demiro\u{g}lu Karabulut, Waring's problem in finite rings, J. Pure Appl. Algebra {\bf 223} (2019), no.~8, 3318--3329

\bibitem{dodson1}
M. M. Dodson, On Waring's problem in $p$-adic fields, Acta Arith. {\bf 22} (1972/73), 315--327

\bibitem{kowmis}
T. Kowalczyk\ and\ P. Miska, On Waring Numbers of Henselian Rings, Mathematika {\bf 70} (2024), no.~4, e12276

\bibitem{voloch}
J. F. Voloch, On the $p$-adic Waring's problem, Acta Arith. {\bf 90} (1999), no.~1, 91--95

\bibitem{wooley}
R. C. Vaughan\ and\ T. D. Wooley, Waring's problem: a survey, in {\it Number theory for the millennium, III (Urbana, IL, 2000)}, 301--340, A K Peters, Natick, MA

\end{thebibliography}

\end{document}